\documentclass[12pt]{amsart}
\usepackage{latexsym,amssymb,amsmath,amsthm,amscd,enumerate,graphicx,esint}

\makeatletter
\@namedef{subjclassname@2010}{%
  \textup{2010} Mathematics Subject Classification}
\makeatother

\newtheorem{thm}{Theorem}[section]

\newtheorem{lem}[thm]{Lemma}
\newtheorem{prop}[thm]{Proposition}

\theoremstyle{definition}
\newtheorem{defin}[thm]{Definition}

\numberwithin{equation}{section}

\newcommand{\cB}{{\mathcal B}}

\newcommand{\cD}{{\mathcal D}}

\newcommand{\cM}{{\mathcal M}}

\newcommand{\cQ}{{\mathcal Q}}
\newcommand{\cS}{{\mathcal S}}

\newcommand{\N}{{\mathbb N}}

\newcommand{\R}{{\mathbb R}}

\def\al{\alpha}
\def\bt{\beta}
\def\gm{\gamma}

\def\kp{\kappa}
\def\lm{\lambda}

\def\sg{\sigma}

\def\Om{\Omega}

\def\0{\emptyset}
\def\1{\textbf{\rm 1}}
\def\6{\partial}
\def\8{\infty}

\def\lt{\left}
\def\rt{\right}

\def\wt{\widetilde}

\DeclareMathOperator*{\essinf}{\,\text{ess\,inf}\,}

\begin{document}

\title
[Fractional operators on weighted Morrey spaces]
{Fractional operators on weighted Morrey spaces}

\author[S.~Nakamura]{Shohei~Nakamura}
\address{Department of Mathematics and Information Sciences, Tokyo Metropolitan University, 1-1 Minami Ohsawa, Hachioji, Tokyo 192-0397, Japan}
\email{pokopoko9131@icloud.com}

\author[Y.~Sawano]{Yoshihiro~Sawano}
\address{Department of Mathematics and Information Sciences, Tokyo Metropolitan University, 1-1 Minami Ohsawa, Hachioji, Tokyo 192-0397, Japan}
\email{ysawano@tmu.ac.jp}

\author[H.~Tanaka]{Hitoshi~Tanaka}
\address{Research and Support Center on Higher Education for the hearing and Visually Impaired, National University Corporation Tsukuba University of Technology, Kasuga 4-12-7, Tsukuba City, Ibaraki, 305-8521 Japan}
\email{htanaka@k.tsukuba-tech.ac.jp}

\thanks{
The second author is partially supported by Grand-in-Aid for 
Scientific Research (C), No. 16K05209, for Japan Society for the Promotion of Science. 
The third author is supported by 
Grant-in-Aid for Scientific Research (C) (15K04918), 
the Japan Society for the Promotion of Science. 
}

\date{}



\begin{abstract}
A necessary condition and a sufficient condition 
for one weight norm inequalities on Morrey spaces to hold are given 
for the fractional maximal operator and the fractional integral operator. 
We clarify the difference 
between the behavior of the fractional maximal operator 
and the one of the fractional integral operator 
which is originated from the structure of Morrey spaces. 
Both the necessary condition
and 
the sufficient condition 
are also verified for the power weights. 
\end{abstract}

\maketitle

\section{Introduction}\label{sec1}
The purpose of this paper is to develop a theory of weights 
for fractional maximal and integral operators
on Morrey spaces. 
There are several results concerning the weight theory on Morrey spaces 
by assuming the $A_p$ conditions (for example, \cite{ISST,NS15}). 
However, it was pointed in \cite{Sam1,Ta} that 
the $A_p$ condition is not suitable for the Morrey setting. 
In fact, it is too strong.
So, the problem we address in is to establish the 
weight theory on Morrey spaces without the $A_p$ condition. 
After C.~Morrey introduced Morrey spaces,
many people realized that
Morrey spaces are used for various purpose.
One of the reasons is that
Morrey spaces
describe local regularity more precisely than Lebesgue spaces.
As a result
we can use Morrey spaces widely 
not only in harmonic analysis 
but also in partial differential equations 
(cf. \cite{GT}). 

To define Morrey spaces,
we shall consider all cubes in $\R^n$ 
which have their sides parallel to the coordinate axes.
We denote by $\cQ$ the family of all such cubes. 
For a cube $Q\in\cQ$ we use 
$\ell(Q)$ to denote the sides length of $Q$, 
$c(Q)$ to denote the center of $Q$, 
$|Q|$ to denote the volume of $Q$ and 
$cQ$ to denote the cube with the same center as $Q$ 
but with side-length $c\ell(Q)$. 
Let $0<p\le p_0<\8$ be two real parameters. 
For $f\in L^p_{\text{loc}}(\R^n)$ 
define 
\begin{equation}\label{Morrey}
\|f\|_{\cM^{p_0}_p}
=
\sup_{Q\in\cQ}
|Q|^{1/p_0}
\lt(\fint_{Q}|f|^p\,dx\rt)^{1/p},
\end{equation}
where we have used a barred integral 
to denote the integral average
\[
\fint_{Q}f\,dx
=
\frac{1}{|Q|}\int_{Q}f\,dx.
\]
The Morrey space $\cM^{p_0}_p(\R^n)$ 
is defined to be the subset 
of all $L^p$ locally integrable functions $f$ on $\R^n$ 
for which $\|f\|_{\cM^{p_0}_p}$ is finite. 
It is easy see that 
$\|\cdot\|_{\cM^{p_0}_p}$
is a norm if $p\ge 1$ and 
is a quasi-norm if $p\in(0,1)$. 
The completeness of Morrey spaces 
follows easily by that of Lebesgue spaces. 
Applying H\"{o}lder's inequality, 
we see also that 
\[
\|f\|_{\cM^{p_0}_{p_1}}
\ge
\|f\|_{\cM^{p_0}_{p_2}}
\text{ for all }
p_0\ge p_1\ge p_2>0.
\]
This tells us that 
\begin{equation}\label{1.1}
L^{p_0}(\R^n)
=\cM^{p_0}_{p_0}(\R^n)
\subset \cM^{p_0}_{p_1}(\R^n)
\subset \cM^{p_0}_{p_2}(\R^n)
\text{ for all }
p_0\ge p_1\ge p_2>0.
\end{equation}
Sometimes it is convenient
to define Morrey spaces in an equivalent form.
Let $1<p<p_0<\8$ and 
define $\lm$ by $\lm/n=1-p/p_0$. 
We will use the notation 
\[
\|f\|_{L^{p,\lm}}
=
\sup_{Q\in\cQ}
\lt(\frac1{|Q|^{\lm/n}}\int_{Q}|f|^p\,dx\rt)^{1/p}
\]
and 
$L^{p,\lm}(\R^n)$ 
to denote 
$\|f\|_{\cM^{p_0}_p}$ 
and 
$\cM^{p_0}_p(\R^n)$, 
respectively. 

As we mentioned above,
Morrey spaces reflect local properties
of the functions.
Due to this property,
we can describe the boundedness property
of the linear (or sublinear) operators
more precisely than Lebesgue spaces.
We envisage the following operators in this paper.
\begin{itemize}
\item
Given $0<\al<n$ and a measurable function $f$, 
we define the fractional integral operator $I_{\al}$ by 
\[
I_{\al}f(x)
=
\int_{\R^n}\frac{f(y)}{|x-y|^{n-\al}}\,dy.
\]
\item
Given $0\le\al<n$ and a measurable function $f$, 
we define the fractional maximal operator $M_{\al}$ by 
\[
M_{\al}f(x)
=
\sup_{Q\in\cQ}
\1_{Q}(x)|Q|^{\al/n}\fint_{Q}|f|\,dy,
\]
where $\1_{Q}$ denotes the characteristic function of the cube $Q$. 
If $\al=0$ we drop the subscript $\al$. 
Thus, $M=M_0$ is the Hardy-Littlewood maximal operator. 
\end{itemize}

Based on the definition above,
let us see 
two remarkable results
asserting for what parameters
$p,p_0,q,q_0$ 
the fractional integral operator
$I_\al$ is bounded
from
$\cM^{p_0}_p$
to
$\cM^{q_0}_q$,
where $1<p\le p_0<\8$ 
and $1<q\le q_0<\8$. 
The first one is due to Spanne 
(unpublished): 
the inequality 
\begin{equation}\label{1.2}
\|I_{\al}f\|_{\cM^{q_0}_q}
\le C
\|f\|_{\cM^{p_0}_p}
\end{equation}
holds if 
\[
\frac1{q_0}=\frac1{p_0}-\frac{\al}n
\text{ and }
\frac1q=\frac1p-\frac{\al}n.
\]
The second one is due to Adams \cite{Ad} 
(see also \cite{CF}): 
the inequality \eqref{1.2} holds if 
\[
\frac1{q_0}=\frac1{p_0}-\frac{\al}n
\text{ and }
\frac{q}{q_0}=\frac{p}{p_0}.
\]
A simple arithmetic shows that 
\[
\frac1p-\frac{\al}n
=
\frac{p_0}p
\lt(
\frac1{p_0}
-
\frac{p}{p_0}\cdot\frac{\al}n
\rt)
\ge
\frac{p_0}p
\lt(
\frac1{p_0}-\frac{\al}n
\rt)
=
\frac{p_0}{pq_0}.
\]
This inequality together with \eqref{1.1} says that
the Spanne target space is larger than the Adams target space.
Thus, we can say that 
Adams improved the result of Spanne. 
Furthermore, 
Olsen \cite{Ol} showed by an example that the result of Adams is optimal. 

By weights we will always mean 
non-negative, 
locally integrable functions 
which are positive on a set of positive measure. 
Given a measurable set $E$ and a weight $w$, 
$w(E)=\int_{E}w$. 
Given $1<p<\8$, $p'=p/(p-1)$ 
will denote the conjugate exponent number of $p$. 

Given $p>1$, 
one says that a weight $w$ on $\R^n$ 
belongs to the Muckenhoupt class $A_p$ 
if
\[
[w]_{A_p}
=
\sup_{Q\in\cQ}
\frac{w(Q)\sg(Q)^{p-1}}{|Q|^p}
<\8,\quad
\sg=w^{1-p'}.
\]
For $p=1$, 
one says that a weight $w$ on $\R^n$ 
belongs to the Muckenhoupt class $A_1$ 
if 
\[
[w]_{A_1}
=
\sup_{Q\in\cQ}
\frac{w(Q)/|Q|}{\essinf_{x\in Q}w(x)}
<\8.
\]
In \cite{Mu}, 
Muckenhoupt showed that, 
for $p>1$, 
the weights satisfying the $A_p$ condition 
are exactly the weights for which 
the Hardy-Littlewood maximal operator $M$ is bounded 
on $L^p(w)$. 

Let $0<\al<n$, 
$1<p<n/\al$ and $q$ be defined by 
$1/q=1/p-\al/n$. 
In \cite{MuWh}, 
Muckenhoupt and Wheeden characterized 
the weighted strong type inequality 
for fractional maximal and integral operators 
in terms of the so-called $A_{p,q}$ condition. 
They showed that the inequality 
\begin{equation}\label{1.3}
\|T_{\al}f\|_{L^q(w^q)}
\le C
\|f\|_{L^p(w^p)},
\end{equation}
where $T_{\al}$ is the operator $I_{\al}$ or $M_{\al}$, 
holds if and only if 
$w\in A_{p,q}$. That is, 
\[
[w]_{A_{p,q}}
=
\sup_{Q\in\cQ}
\lt(\fint_{Q}w^q\rt)^{1/q}
\lt(\fint_{Q}w^{-p'}\rt)^{1/p'}
<\8.
\]

If $p>1$, we have that 
$w\in A_{p,q}$ 
if and only if 
$w^q\in A_{1+q/p'}$; 
this follows at once from the definition. 

The following is the sharp weighted bound
for the fractional integral operator $I_\alpha$.

\begin{thm}[\text{\cite[Theorem 2.6]{LMPT}}]
\label{thm1.1}
Let $w\in A_{p,q}$. 
Let $0<\al<n$, 
$1<p<n/\al$ and $q$ be defined by 
$1/q=1/p-\al/n$. Then 
\[
\|I_{\al}\|_{L^p(w^p)\to L^q(w^q)}
\le C
[w]_{A_{p,q}}^{(1-\al/n)\max\{q,p'\}}.
\]
Furthermore, the power $(1-\al/n)\max\{q,p'\}$ is sharp.
\end{thm}

For $E\subset\R^n$ and $0<\al\le n$, 
the $\al$-dimensional Hausdorff content of $E$ is defined by
\[
H^{\al}(E)
=
\inf\lt\{\sum_jl(Q_j)^{\al}\rt\},
\]
where the infimum is taken over all coverings of $E$ 
by countable families of cubes 
$\{Q_j\}\subset\cQ$. 
The Choquet integral of $\phi\ge 0$ 
with respect to the Hausdorff content $H^{\al}$ is defined by 
\[
\int_{\R^n}\phi\,dH^{\al}
=
\int_0^{\8}
H^{\al}(\{y\in\R^n:\,\phi(y)>t\})
\,dt.
\]

\begin{defin}\label{def1.2}
Let $0<\lm<n$. Define 
the basis $\cB_{\lm}$ 
to be the set of all weights $b$ such that 
$b\in A_1$ and 
$\int_{\R^n}b\,dH^{\lm}\le 1$.
\end{defin}

Let $1<p<p_0<\8$ and 
set $\lm/n=1-p/p_0$. 
Then one has 
(see \cite{AX2} and also \cite{Ta})
\begin{equation}\label{1.4}
\|f\|_{\cM^{p_0}_p}
=
\|f\|_{L^{p,\lm}}
\approx
\sup_{b\in\cB_{\lm}}
\lt(\int_{\R^n}|f|^pb\,dx\rt)^{1/p}.
\end{equation}

\begin{defin}\label{def1.3}
Let $1<p<\8$ and $0<\lm<n$. 
The space $H^{p,\lm}(\R^n)$ 
is defined by the set of all measurable functions $f$ on $\R^n$ 
with the quasi norm 
\[
\|f\|_{H^{p,\lm}}
=
\inf_{b\in\cB_{\lm}}
\lt(\int_{\R^n}|f|^pb^{1-p}\,dx\rt)^{1/p}
<\8.
\]
\end{defin}

Let $1<p<p_0<\8$ and 
set $\lm/n=1-p/p_0$. 
For any $b\in\cB_{\lm}$ and 
for all non-negative functions 
$f\in L^{p,\lm}(\R^n)$ 
and 
$g\in H^{p',\lm}(\R^n)$, 
by H\"{o}lder's inequality that 
\begin{align*}
\int_{\R^n}fg\,dx
=
\int_{\R^n}fb^{1/p}gb^{-1/p}\,dx
\le
\lt(\int_{\R^n}f^pb\,dx\rt)^{1/p}
\lt(\int_{\R^n}g^{p'}b^{1-p'}\,dx\rt)^{1/p'},
\end{align*}
which implies by \eqref{1.4} 
H\"{o}lder's inequality for Morrey spaces
\begin{equation}\label{1.5}
\int_{\R^n}fg\,dx
\le C
\|f\|_{L^{p,\lm}}
\|g\|_{H^{p',\lm}}.
\end{equation}

In this paper 
we shall establish the following theorems:

\begin{thm}\label{thm1.4}
Let 
$0\le\al<n$, 
$1<p<p_0<\8$, 
$1<q<q_0<\8$ 
and $w$ be a weight. 
Suppose that 
\[
\frac1{q_0}=\frac1{p_0}-\frac{\al}n
\text{ and }
\frac{q}{q_0}=\frac{p}{p_0}.
\]
Set 
$\lm/n=1-p/p_0=1-q/q_0$. 
Consider the following three statements\text{:}

\begin{itemize}
\item[\text{(a)}] 
There exists a constant $C_1>0$ 
such that 
\[
\|(M_{\al}f)w\|_{L^{q,\lm}}
\le C_1
\|fw\|_{L^{p,\lm}}
\]
holds for every measurable function $f$ with 
$fw\in L^{p,\lm}(\R^n)$;
\item[\text{(b)}] 
There exists a constant $C_2>0$ 
such that 
\begin{equation}\label{1.6}
\sup_{Q\in\cQ}
|Q|^{\al/n-1}
\|w\1_{Q}\|_{L^{q,\lm}}
\|w^{-1}\1_{Q}\|_{H^{p',\lm}}
\le C_2;
\end{equation}
\item[\text{(c)}] 
For any $Q_0\in\cQ$, 
there exists 
$b_{Q_0}\in\cB_{\lm}$ 
satisfying the following\text{:}
\begin{equation}\label{1.7}
\sup_{\substack{Q\in\cQ \\ Q\subset Q_0}}
\lt(\fint_{Q}w^q\,dx\rt)^{1/q}
\lt(\fint_{Q}[wb_{Q_0}^{1/p}]^{-p'}\,dx\rt)^{1/p'}
\le C_3
\ell(Q_0)^{\lm/p}
\end{equation}
and 
\begin{equation}\label{1.8}
[wb_{Q_0}^{1/p}]_{A_s}\le C_3
\text{ for some }s\ge 1,
\end{equation}
where the constant $C_3$ is independent of the choices $Q_0$.
\end{itemize}

\noindent
Then,

\begin{itemize}
\item[\text{(I)}] 
That \text{(a)} implies \text{(b)}
with $C_2\le C C_1$;
\item[\text{(II)}] 
Those \text{(b)} and \text{(c)} imply \text{(a)} 
with 
$C_1\le C
(C_2^{(q-p)/q}
C_3^{(p+1)/q}
+C_2)$. 
\end{itemize}
\end{thm}

Unfortunately, 
because of the additional condition $(c)$, 
Theorem \ref{thm1.4} does not completely characterize the 
boundedness of $M_{\al}$ on weighted Morrey spaces. 
However, by employing Theorem \ref{thm1.4}, 
we can still settle down the problem 
at least for power weights; 
see Proposition \ref{pr-power}.

For the fractional integral operator $I_\al$, 
we have the following.
\begin{thm}\label{thm1.5}
Let 
$0<\al<n$, 
$1<p<p_0<\8$, 
$1<q<q_0<\8$ 
and $w$ be a weight. 
Suppose that 
\[
\frac1{q_0}=\frac1{p_0}-\frac{\al}n
\text{ and }
\frac{q}{q_0}=\frac{p}{p_0}.
\]
Set 
$\lm/n=1-p/p_0=1-q/q_0$. 
Consider the following three statements\text{:}

\begin{itemize}
\item[\text{(a)}] 
There exists a constant $C_1>0$ 
such that 
\[
\|(I_{\al}f)w\|_{L^{q,\lm}}
\le C_1
\|fw\|_{L^{p,\lm}}
\]
holds for every function $f$ with 
$fw\in L^{p,\lm}(\R^n)$;
\item[\text{(b)}] 
There exists a constant $C_2>0$ 
such that 
\[
\|(M_{\al}f)w\|_{L^{q,\lm}}
\le C_2
\|fw\|_{L^{p,\lm}}
\]
holds for every function $f$ with 
$fw\in L^{p,\lm}(\R^n)$;
\item[\text{(c)}] 
There exists $\kp>1$ such that 
\begin{equation}\label{1.9}
2\|w\1_{Q}\|_{L^{q,\lm}}
\le
\|w\1_{\kp Q}\|_{L^{q,\lm}}
\end{equation}
holds for every $Q\in\cQ$.
\end{itemize}

\noindent
Then,

\begin{itemize}
\item[\text{(I)}] 
That \text{(a)} implies \text{(b)} and \text{(c)} 
with $C_2\le C C_1$;
\item[\text{(II)}] 
Those \text{(b)} and \text{(c)} imply \text{(a)} 
with 
$C_1\le C C_2^{q+1}$. 
\end{itemize}
\end{thm}

Theorem \ref{thm1.5} implies that the boundedness of 
$I_{\al}$ is equivalent to 
the one of $M_{\al}$ and 
the additional condition \eqref{1.9}. 
Note that the additional condition \eqref{1.9} was introduced in 
\cite{NS15} as the weighted integral condition to ensure 
the boundedness of the singular integral operator on weighted Morrey spaces.

It is well known that Muckenhoupt introduced
the class of weight $A_p$ in his paper
\cite{Mu,MuWh}.
In fact Muckenhoupt was successfull
in characterizing the condition
for $M$ to be bounded
on $L^p(w)$.
In establishing the theory of weights for 
Lebesgue spaces,
it is difficult to obtain the strong $A_p$
estimates.
Muckenhoupt established the strong weight theory
in \cite[Section 4]{Mu} for $n=1$
and
Coifman and Fefferman considered
the higher dimensional case
\cite{CoFe74}.
We can say that the key tool
is the Calder\''{o}n-Zygmund decomposition.
The Calder\''{o}n-Zygmund decomposition is skillfully used
to solve the $A_2$ conjecture
\cite{Hy3}
and develop a modern weighted theory
\cite{HPR12,HyPe13}.
However, it seems that the Calder\'{o}n-Zygmund theory
is not enough
when we prove the boundedness of the operators
on Morrey spaces.
A standard technique
to prove the boundedness of the operators
on Morrey spaces
is to fix a cube $Q$,
as is seen from the definition
(\ref{Morrey}).
Accordingly, when we are given a function
$f$, we decompose it according to $3Q$.
Let $f_1=f\chi_{3Q}$ and $f_2=f-f_1$.
Then we can benefit a lot
from the Calder\'{o}n-Zygmund theory
for the function $f_1$.
However, it seems that some different approaches
are necessary for $f_2$.
In this paper,
we applied this strategy
in the proof of Lemmma \ref{lem:5.2}.
See
(\ref{5.4})
and
(\ref{5.6})
for the estimates
for $f_1$ and $f_2$,
respectively,
where
a special tool
(\ref{5.5})
is necessary 
for $f_2$
in order to do without the Calder\'{o}n-Zygmund decomposition.

One of the striking achievement
in the theory of weighted Lebesgue spaces
is that the classes
$\{A_p\}_{p>1}$ enjoys the openness property.
Originally,
Muckenhoupt used
to show that the strong boundedness
on $L^p(w)$ is equivalent to $w \in A_p$
\cite{Mu}.
From the definition of $A_p$,
we can show that $M$ is weak bounded on $L^p(w)$
using the covering lemma.
It is not so hard to show
that
the strong boundedness
on $L^p(w)$ implies $w \in A_p$.
We follow the same line in 
our proof of 
(\ref{1.6}) based on the strong boundedness
of Morrey spaces in Theorem \ref{thm1.4}.
However,
even in the case of Lebesgue spaces,
it \mbox{\it was} hard to show that $w \in A_p$ implies
the strong boundedness
on $L^p(w)$.
In fact, the proof hinged upon the openness
property asserting that $w \in A_{q}$
for some $1<q<p$.
Since $M$ is weak bounded on $L^q(w)$
and bound on $L^\8$ trivially,
we see that $M$ is bounded on $L^p(w)$.
When we want to run this program,
we are faced with the problem of showing
the weak boundedness on weighted Morrey spaces
although we still have some openness property, see
\cite{LO};
once again, the Calder\'{o}n-Zygmund decomposition
is not enough.
We remark that the results
in \cite{LO} are available in
weighted Morrey spaces
by reexamining the proof.

Although the openness property
seems to have been essential in early 80's,
it turned out that  we can prove the 
$L^p(w)$ boundedness of $M$
without using the openness property
\cite{HKN,Jawerth,Lerner}.
Among others,
Lerner used a universal estimate 
(\ref{5.2})
for the weighted dyadic Hardy-Littlewood
maximal operator.
His main idea is to convert
the Hardy-Littlewood maximal operator
adapted to the weighted Lebesgue space $L^p(w)$
\cite[p. 2831]{Lerner}.
Although we still have a counterpart
to weighted Morrey spaces
of the universal estimate Lerner used,
the gap exists
between the condition
(\ref{1.6}) and
the universal estimate we obtain,
see Lemma \ref{lem:5.2}.

Another barrier for us to study Morrey spaces
is that Morrey spaces are not rearrangement invariant
as is seen from the example in \cite[Proposition 4.1]{SST11}.
In fact, another example shows that
the Morrey space $L^{p,\lambda}({\mathbb R}^n)$
with $1<\lambda<n$
is not embedded 
into $L^1({\mathbb R}^n)+L^\infty({\mathbb R}^n)$,
see \cite[Section 6]{HS}.
This fact prevents us from using
the theory developed in \cite[Theorem 2.4]{LePe07}.
Since Morrey spaces are not rearrangement invariant,
it is convenient for us to use the decreasing rearrangement.

We can locate the function space
$H^{p,\lambda}({\mathbb R}^n)$
as a new tool to overcome these problems.

Here and below,
the letter $C$ will be used for constants 
that may change from one occurrence to another. 
Constants with subscripts, such as $C_1$, $C_2$, do not change 
in different occurrences. 
By $A\approx B$ we mean that 
$c^{-1}B\le A\le cB$ 
with some positive constant $c$ independent of appropriate quantities.

\section{Proof of Theorem \ref{thm1.4}}\label{sec2}
In what follows 
we shall prove Theorem \ref{thm1.4}. 
We need three lemmas 
(cf. \cite{Ta} for the first lemma).

\begin{lem}\label{lem2.1}
Let $1<p<p_0<\8$ and 
set $\lm/n=1-p/p_0$. 
Then, 
for any measurable function $g$ on $\R^n$,
we have the estimate $($allowing to be infinite$)$ 
\[
\|g\|_{H^{p',\lm}}
\approx
\sup_{f}\int_{R^n}|fg|\,dx,
\]
where the supremum is taken over all functions 
$f\in L^{p,\lm}(\R^n)$ with unit norm.
\end{lem}

\begin{lem}\label{lem2.2}
Let $w\in A_p$, $p\ge 1$, and $Q\in\cQ$. 
Then, 
for any measurable set $S\subset Q$, 
\[
\lt(\frac{|S|}{|Q|}\rt)^pw(Q)
\le C
[w]_{A_p}w(S).
\]
\end{lem}

\begin{proof}
Using the well-known fact that 
\[
\sup_{t>0}
t^pw(\{x:\,Mf(x)\ge t\})
\le C
[w]_{A_p}\|f\|_{L^p(w)}^p,
\]
we have that 
\[
w(Q)
\le
w(\{x:\,M[\1_{S}](x)\ge|S|/|Q|\})
\le C
[w]_{A_p}
(|S|/|Q|)^{-p}w(S),
\]
which proves the lemma.
\end{proof}

To describe the third lemma,
we need terminology.
We say that a family $\cS$ of cubes from $\R^n$ is $\eta$ sparse, 
$0<\eta<1$, 
if for every $Q\in\cS$, 
there exists a measurable set 
$E_{Q}\subset Q$ such that 
$|E_{Q}|\ge\eta|Q|$, and 
the sets $\{E_{Q}\}_{Q\in\cS}$ 
are pairwise disjoint. 
Given a cube $Q_0\in\cQ$, 
let $\cD(Q_0)$ denote the set of all dyadic cubes 
with respect to $Q_0$, that is, 
the cubes obtained by repeated subdivision of $Q_0$ 
and each of its descendants into $2^n$ congruent subcubes. 
By convention $Q_0$ itself belongs to $\cD(Q_0)$. 

\begin{lem}\label{lem2.3}
Let $0\le\al<n$. 
Suppose that 
the non-negative and bounded
function $f$ has compact support. 
Then, for any cube $Q_0\in\cQ$, 
there exists a $1/2$ sparse family 
$\cS\subset\cD(Q_0)$ 
such that, for all $x\in Q_0$, 
\[
M_{\al}f(x)
\le C
L_{\al}^{\cS}f(x)
+
c_{\8},
\]
where 
\[
L_{\al}^{\cS}f(x)
=
\sum_{Q\in\cS}
\1_{E_{Q}}(x)
|Q|^{\al/n}\fint_{3Q}f\,dy
\]
and 
\[
c_{\8}
=
\sup_{\substack{Q\in\cQ \\ Q\supset Q_0}}
|Q|^{\al/n}\fint_{Q}f\,dx.
\]
\end{lem}

\begin{proof}
Fix $Q_0\in\cQ$.
We write 
\[
\wt{M}_{\al}f(x)
=
\sup_{Q\in\cD(Q_0)}
\1_{Q}(x)
|Q|^{\al/n}\fint_{3Q}f\,dy.
\]
It is easy see that, 
for all $x\in Q_0$, 
\[
M_{\al}f(x)
\le C
\wt{M}_{\al}f(x)
+
c_{\8}.
\]
Let 
$a_0=|Q_0|^{\al/n}\fint_{3Q_0}f\,dx$ 
and 
$a=9^n2^{n+1-\al}$. 
For each $k=0,1,2,\ldots$, 
define 
\[
D_k
=
\bigcup\lt\{
Q\in\cD(Q_0):\,
|Q|^{\al/n}\fint_{3Q}f\,dx
\ge
a_0a^k
\rt\}
(\subset Q_0).
\]
Considering the maximal cubes with respect to inclusion,
we can write 
\[
D_k=\bigcup_jQ_j^k,
\]
where the cubes $\{Q_j^k\}$ are pairwise disjoint.
By the maximality of $Q_j^k$, 
we see that 
\begin{equation}\label{2.1}
a_0a^k
\le
|Q_j^k|^{\al/n}\fint_{3Q_j^k}f\,dx
\le 2^{n-\al}
a_0a^k.
\end{equation}
We shall verify that 
the family $\cS=\{Q_j^k\}$
is $1/2$ sparse. 
To this end, we let 
\[
E_{Q_j^k}
=
Q_j^k\setminus D_{k+1},
\]
then we see that 
the sets $\{E_{Q_j^k}\}$ 
are pairwise disjoint and 
decompose $Q_0$. 
So, we need only verify that 
\begin{equation}\label{2.2}
|E_{Q_j^k}|
\ge
\frac12|Q_j^k|.
\end{equation}
Notice that, 
if $Q_i^{k+1}\subset Q_j^k$, 
then by \eqref{2.1} 
\begin{align*}
a_0a^{k+1}
&\le
|Q_i^{k+1}|^{\al/n}
\fint_{3Q_i^{k+1}}f\,dx
<
|Q_j^k|^{\al/n}
\fint_{3Q_i^{k+1}}f\,dx
\\ &\le
|Q_j^k|^{\al/n}
M[f\1_{3Q_j^k}](x)
\text{ for all }
x\in Q_i^{k+1}.
\end{align*}
This entails 
\[
Q_j^k\cap D_{k+1}
\subset
\lt\{x\in\R^n:\,
M[f\1_{3Q_j^k}](x)
\ge
\frac{a_0a^{k+1}}{|Q_j^k|^{\al/n}}
\rt\}.
\]
The weak-$(1,1)$ boundedness of $M$ 
together with \eqref{2.1} yields 
\begin{align*}
|Q_j^k\cap D_{k+1}|
&\le
3^n
\cdot
\frac{|Q_j^k|^{\al/n}}{a_0a^{k+1}}
\cdot
\int_{3Q_j^k}f
\le
3^n
\cdot
2^{n-\al}a_0a^k
\cdot
\frac{|Q_j^k|^{\al/n}}{a_0a^{k+1}}
\cdot
\frac{3^n|Q_j^k|}{|Q_j^k|^{\al/n}}
\\ &=
\frac{9^n2^{n-\al}}{a}
|Q_j^k|
=
\frac12|Q_j^k|,
\end{align*}
which implies \eqref{2.2}.

Finally, 
for each $Q=Q_j^k\in\cS$ 
and any $x\in E_{Q}$, 
we have by \eqref{2.1} that 
\[
\wt{M}_{\al}f(x)
\le
a_0a^{k+1}
\le a
|Q|^{\al/n}\fint_{3Q}f\,dy.
\]
Since 
the sets $\{E_{Q}\}_{Q\in\cS}$
are pairwise disjoint and 
decompose $Q_0$, 
we conclude that 
\[
\wt{M}_{\al}f(x)
\le C
L_{\al}^{\cS}f(x)
\text{ for all }
x\in Q_0.
\]
This completes the proof. 
\end{proof}

\subsection{Proof of Theorem \ref{thm1.4} (I)}\label{ssec2.1}
Assume the statement \text{(a)}. 
Then the inequality 
\[
\|(M_{\al}f)w\|_{L^{q,\lm}}
\le C_1
\|fw\|_{L^{p,\lm}}
\]
holds for every function $f$ with 
$fw\in L^{p,\lm}(\R^n)$.
For any cube $Q\in\cQ$ and 
any function $f$ with 
$fw\in L^{p,\lm}(\R^n)$, 
\[
|Q|^{\al/n-1}
\int_{Q}|f|\,dx
\times
\|w\1_{Q}\|_{L^{q,\lm}}
\le
\|M_{\al}[f\1_{Q}]w\|_{L^{q,\lm}}
\le C_1
\|fw\1_{Q}\|_{L^{p,\lm}}.
\]
Taking the supremum over all functions $f$ 
with 
$\|fw1_{Q}\|_{L^{p,\lm}}\le 1$, 
we have by Lemma \ref{lem2.1} 
\[
|Q|^{\al/n-1}
\|w\1_{Q}\|_{L^{q,\lm}}
\|w^{-1}\1_{Q}\|_{H^{p',\lm}}
\le C C_1,
\]
which is the statement {\text(b)}.

\subsection{Proof of Theorem \ref{thm1.4} (II)}\label{ssec2.2}
To prove sufficiency we may assume that 
the function $f$ is non-negative and bounded and 
that $f$ has compact support. 
Fix $Q_0\in\cQ$. 
We have to evaluate the quantity 
\[
\lt(\frac1{|Q_0|^{\lm/n}}\int_{Q_0}[(M_{\al}f)w]^q\,dx\rt)^{1/q}.
\]
By Lemma \ref{lem2.3} 
we can select a $1/2$ sparse family 
$\cS\subset\cD(Q_0)$ 
such that, for all $x\in Q_0$, 
\[
M_{\al}f(x)
\le C
L_{\al}^{\cS}f(x)
+
c_{\8},
\]
where 
\[
L_{\al}^{\cS}f(x)
=
\sum_{Q\in\cS}
\1_{E_{Q}}(x)|Q|^{\al/n}\fint_{3Q}f\,dy
\]
and 
\[
c_{\8}
=
\sup_{\substack{Q\in\cQ \\ Q\supset Q_0}}
|Q|^{\al/n}\fint_{Q}f\,dx.
\]

We first estimate 
\[
\text{(i)}
=
c_{\8}
\lt(\frac1{|Q_0|^{\lm/n}}\int_{Q_0}w^q\,dx\rt)^{1/q}.
\]
For any cube $Q\supset Q_0$, 
\begin{align*}
\lefteqn{
|Q|^{\al/n}\fint_{Q}f\,dx
\lt(\frac1{|Q_0|^{\lm/n}}\int_{Q_0}w^q\,dx\rt)^{1/q}
}\\ &\le
|Q|^{\al/n-1}
\|w\1_{Q}\|_{L^{q,\lm}}
\int_{Q}w^{-1}\cdot fw\,dx
\\ &\le C
|Q|^{\al/n-1}
\|w\1_{Q}\|_{L^{q,\lm}}
\|w^{-1}\1_{Q}\|_{H^{p',\lm}}
\|fw\|_{L^{p,\lm}}
\\ &\le C C_2
\|fw\|_{L^{p,\lm}},
\end{align*}
where we have used 
H\"{o}lder's inequality \eqref{1.5} 
and \eqref{1.6}. 
This implies, 
since the cube $Q\supset Q_0$ is arbitrary, 
\[
\text{(i)}
\le C C_2
\|fw\|_{L^{p,\lm}}.
\]

We next estimate 
\[
\text{(ii)}
=
\int_{Q_0}[(L_{\al}^{\cS}f)w]^q\,dx.
\]
Take $b=b_{9Q_0}\in\cB_{\lm}$ 
satisfying 
\eqref{1.7} and \eqref{1.8} 
(replacing $Q_0$ with $9Q_0$). 
Set $u=w^q$ and 
$\sg=[wb^{1/p}]^{-p'}$. 
Since the sets $E_{Q}$, 
$Q\in\cS$, 
are pairwise disjoint, we have that 
\[
\text{(ii)}
=3^{-nq}
\sum_{Q\in\cS}
\lt(|Q|^{\al/n-1}\int_{3Q}f\,dx\rt)^q
u(E_{Q}).
\]

We recall the following\text{:}
Since
\[
\frac{\int_{3Q}f\,dx}{\sg(9Q)}
=
\frac{\int_{3Q}f\sg^{-1}\,d\sg}{\sg(9Q)}
\le
\inf_{y\in Q}M_{\sg}^c[f\sg^{-1}](y),
\]
where $M_{\sg}^c$ is 
the centered weighted Hardy-Littlewood maximal operator 
with respect to $\sg$, 
we obtain 
\begin{align*}
\sum_{Q\in\cS}
\lt(\frac{\int_{3Q}f}{\sg(9Q)}\rt)^p
\sg(E_{Q})
&\le
\sum_{Q\in\cS}
\int_{E_{Q}}
M_{\sg}^c[f\sg^{-1}]^p
\,d\sg
\\ &\le
\|M_{\sg}^c[f\sg^{-1}]\|_{L^p(\sg)}^p
\le 
\left(\frac{p}{p-1}\right)^p
\|f\sg^{-1}\|_{L^p(\sg)}^p
\\ &=
\left(\frac{p}{p-1}\right)^p
\|fwb^{1/p}\|_{L^p}^p
\le C
\|fw\|_{L^{p,\lm}}^p,
\end{align*}
where we have used \eqref{1.4} and 
the well-known fact that 
$M_{\sg}^c$ is bounded on $L^p(\sg)$. 

With this in mind, 
we shall estimate the quantity 
\[
\text{(iii)}
=
\lt(|Q|^{\al/n-1}\int_{3Q}f\,dx\rt)^q
\frac{u(E_{Q})}{\sg(E_{Q})}.
\]
To this end, we first define
\[
X=\text{(iii)}\cdot\sg(E_{Q})
\text{ and }
Y=\text{(iii)}\cdot\sg(E_{Q})^{1-q/p}.
\]
Then an arithmetic shows that
\[
\text{(iii)}
=
X^{1-p/q}Y^{p/q}.
\]

It follows that 
\begin{align*}
X
&=
u(Q)
\lt(|Q|^{\al/n-1}\int_{3Q}f\,dx\rt)^q
\\ &=
|Q|^{\lm/n}
\lt(
|Q|^{\al/n-1}
\lt(\frac{u(Q)}{|Q|^{\lm/n}}\rt)^{1/q}
\int_{3Q}f\,dx
\rt)^q
\\ &\le C
|Q|^{\lm/n}
\lt(
|Q|^{\al/n-1}
\|w\1_{Q}\|_{L^{q,\lm}}
\|w^{-1}\1_{3Q}\|_{H^{p',\lm}}
\cdot
\|fw\1_{3Q}\|_{L^{p,\lm}}
\rt)^q
\\ &\le C
|Q|^{\lm/n}
\lt(
C_2\|fw\|_{L^{p,\lm}}
\rt)^q,
\end{align*}
where we have used 
H\"{o}lder's inequality \eqref{1.5} 
and our assumption \eqref{1.6}. 

It follows also that 
\begin{align*}
Y
&=
\sg(E_{Q})^{-q/p}
\lt(
|Q|^{\al/n-1}
u(Q)^{1/q}\int_{3Q}f\,dx
\rt)^q
\\ &=
\sg(E_{Q})^{-q/p}
\lt(
|Q|^{\al/n-1}
u(Q)^{1/q}\sg(9Q)
\frac{\int_{3Q}f\,dx}{\sg(9Q)}
\rt)^q
\\ &\le
\lt\{
\lt(\frac{\sg(9Q)}{\sg(E_{Q})}\rt)^{1/p}
\cdot
|Q|^{\al/n-1}
u(9Q)^{1/q}\sg(9Q)^{1/p'}
\cdot
\frac{\int_{3Q}f\,dx}{\sg(9Q)}
\rt\}^q.
\end{align*}

By \eqref{1.8} together with 
$|9Q|=9^n|Q|\le 2\cdot 9^n|E_{Q}|$,
Lemma \ref{lem2.2} gives 
\[
\lt(\frac{\sg(9Q)}{\sg(E_{Q})}\rt)^{1/p}
\le C
C_3^{1/p}.
\]

Meanwhile,
an arithmetic shows that
\begin{align*}
\lefteqn{
(|Q|^{\lm/n})^{1-p/q}
\cdot
\lt(
|Q|^{\al/n-1}
u(9Q)^{1/q}\sg(9Q)^{1/p'}
\rt)^p
}\\ &=
\lt(
|Q|^{(\lm/n)(1/p-1/q)}
|Q|^{\al/n-1}
u(9Q)^{1/q}\sg(9Q)^{1/p'}
\rt)^p
\\ &\qquad\text{ by using }
\frac{\lm}n=1-\frac{p}{p_0}=1-\frac{q}{q_0}
\\ &=
\lt(
|Q|^{(1/p-1/p_0)+(1/q_0-1/q)}
|Q|^{\al/n-1}
u(9Q)^{1/q}\sg(9Q)^{1/p'}
\rt)^p
\\ &\qquad\text{ by using }
\frac1{q_0}-\frac1{p_0}+\frac{\al}n=0
\\ &=
\lt(
|Q|^{-1/q-1/p'}
u(9Q)^{1/q}\sg(9Q)^{1/p'}
\rt)^p
\\ &=C
\lt\{
\lt(\frac{u(9Q)}{|9Q|}\rt)^{1/q}
\lt(\frac{\sg(9Q)}{|9Q|}\rt)^{1/p'}
\rt\}^p
\le C C_3^p
\ell(Q_0)^{\lm},
\end{align*}
where we have used \eqref{1.7}
for the last inequality. 

Altogether, 
\begin{align*}
\ell(Q_0)^{-\lm}
\cdot
\text{(ii)}
&\le C
C_2^{q-p}
C_3^{p+1}
\|fw\|_{L^{p,\lm}}^{q-p}
\sum_{Q\in\cS}
\lt(\frac{\int_{3Q}f}{\sg(9Q)}\rt)^p
\sg(E_{Q})
\\ &\le C
C_2^{q-p}
C_3^{p+1}
\|fw\|_{L^{p,\lm}}^q.
\end{align*}
This proves sufficiency. 

\section{Proof of Theorem \ref{thm1.5}}\label{sec3}
In what follows 
we shall prove Theorem \ref{thm1.5}. 
We need a lemma which is similar to 
Lemma \ref{2.2}. 

\begin{lem}\label{lem3.1}
Let $0<\al<n$ and 
$\kp>1$. 
Suppose that 
the function $f$ is non-negative and bounded and 
that $f$ has compact support. 
Then, for any cube $Q_0\in\cQ$, 
there exists a $1/2$ sparse family 
$\cS\subset\cD(Q_0)$ 
such that, for all $x\in Q_0$, 
\[
I_{\al}f(x)
\le C
\lt(
I_{\al}^{\cS}f(x)
+
C_{\8}
\rt),
\]
where 
\[
I_{\al}^{\cS}f(x)
=
\sum_{Q\in\cS}
\1_{Q}(x)
|Q|^{\al/n}\fint_{3Q}f\,dy
\]
and 
\[
C_{\8}
=
\sum_{k=0}^{\8}
|\kp^kQ_0|^{\al/n}\fint_{\kp^kQ_0}f\,dx.
\]
\end{lem}

\begin{proof}
For all $x\in Q_0$ it follows that 
\[
I_{\al}f(x)
\le C
\lt(
I_{\al}^{\cD(Q_0)}f(x)
+
C_{\8}
\rt),
\]
where 
\[
I_{\al}^{\cD(Q_0)}f(x)
=
\sum_{Q\in\cD(Q_0)}
\1_{Q}(x)|Q|^{\al/n}\fint_{3Q}f\,dy.
\]
Indeed, 
for $x,y\in Q_0$ with $x\neq y$, 
we notice that 
\[
\sum_{\substack{Q\in\cD(Q_0) \\ Q\ni x, 3Q\ni y}}
|Q|^{\al/n-1}
\approx
\frac1{|x-y|^{n-\al}}.
\]
This implies together with Fubini's theorem 
\[
\int_{3Q_0}
\frac{f(y)}{|x-y|^{n-\al}}\,dy
\approx
I_{\al}^{\cD(Q_0)}f(x).
\]
We have also that 
\[
\int_{\R^n\setminus 3Q_0}
\frac{f(y)}{|x-y|^{n-\al}}\,dy
\le C
C_{\8}.
\]

We now construct the sparse set $\cS$.
Let 
$a_0=\fint_{3Q_0}f\,dx$ 
and fix $a=9^n2^{n+1}$. 
For each $k=0,1,2,\ldots$, 
define 
\[
D_k
=
\bigcup\lt\{
Q\in\cD(Q_0):\,
\fint_{3Q}f\,dx
\ge
a_0a^k
\rt\}.
\]
Considering the maximal cubes with respect to inclusion,
we can write 
\[
D_k=\bigcup_jQ_j^k,
\]
where the cubes $\{Q_j^k\}$ are pairwise disjoint.
By the maximality of $Q_j^k$, 
\begin{equation}\label{3.1}
a_0a^k
\le
\fint_{3Q_j^k}f\,dx
\le 2^n
a_0a^k.
\end{equation}
By the same way as the proof of Lemma \ref{lem2.3}, 
letting 
\[
E_{Q_j^k}
=
Q_j^k\setminus D_{k+1},
\]
we can verify that 
the family $\cS=\{Q_j^k\}$
is $1/2$ sparse. 

Finally, if we let 
\[
\cD_j^k
=
\lt\{Q\in\cD(Q_0):\,
Q\subset Q_j^k,\,
a_0a^k
\le
\fint_{3Q}f\,dx
<
a_0a^{k+1}
\rt\},
\]
then we see that 
\[
\cD(Q_0)
=
\bigcup_{k,j}\cD_j^k.
\]
For all $x\in Q_j^k$ 
\begin{align*}
\lefteqn{
\sum_{Q\in\cD_j^k}
\1_{Q}(x)
|Q|^{\al/n}\fint_{3Q}f\,dy
}\\ &\le
a_0a^{k+1}
\sum_{Q\in\cD_j^k}
\1_{Q}(x)|Q|^{\al/n}
\le C
a_0a^{k+1}
|Q_j^k|^{\al/n}
\le a
|Q_j^k|^{\al/n}\fint_{3Q_j^k}f\,dy,
\end{align*}
where we have used \eqref{3.1}.
Thus, 
for all $x\in Q_0$, 
\[
I_{\al}^{\cD(Q_0)}f(x)
\le a
\sum_{Q\in\cS}
\1_{Q}(x)
|Q|^{\al/n}\fint_{3Q}f\,dy.
\]
This complete the proof. 
\end{proof}

\subsection{Proof of Theorem \ref{thm1.5} (I)}\label{ssec3.1}
Assume that \text{(a)} holds. 
The assertion \text{(b)} follows from 
the pointwise inequality 
$M_{\al}f(x)\le C|I_{\al}f(x)|$. 
To prove \text{(c)}, 
we shall obtain the contradiction. 
So, we assume that \eqref{1.9} fails. 
Then, 
for any $m\in\N$, 
there exists $Q_m\in\cQ$ such that 
\begin{equation}\label{3.2}
2\|w\1_{Q_m}\|_{L^{q,\lm}}
>
\|w\1_{mQ_m}\|_{L^{q,\lm}}.
\end{equation}
Now we define for $m>2$ 
\[
f_m(y)
=
\frac
{\1_{mQ_m\setminus 2Q_m}(y)}
{|y-c(Q_m)|^\al}.
\]
Then we notice that, 
for any $x\in Q_m$, 
\[
I_{\al}f_m(x)
\ge C
\int_{mQ_m\setminus 2Q_m}
\frac{dy}{|y-c(Q_m)|^n}
\approx
\int_{\ell(Q_m)}^{m\ell(Q_m)}\frac{dt}t
\approx
\log m.
\]
This implies by \text{(a)} 
\begin{equation}\label{3.3}
\log m
\|w\1_{Q_m}\|_{L^{q,\lm}}
\le C
\|(I_{\al}f_m)w\|_{L^{q,\lm}}
\le C
\|f_mw\|_{L^{p,\lm}}.
\end{equation}
We recall that 
\[
\frac1{q_0}=\frac1{p_0}-\frac{\al}n
\text{ and }
\frac{q}{q_0}=\frac{p}{p_0}.
\]
If we define $r$ by 
$1/p=1/q+1/r$, then 
\[
\frac1r
=
\frac1p-\frac1q
=
\frac{p_0}p
\lt(\frac1{p_0}-\frac1{q_0}\rt)
>
\lt(\frac1{p_0}-\frac1{q_0}\rt)
=
\frac{\al}n,
\]
which means $1<r<n/\al$. 
By H\"{o}lder's inequality 
for the Morrey norms with exponents 
$1/p=1/q+1/r$ and 
$1/p_0=1/q_0+\al/n$, 
\begin{equation}\label{3.4}
\|f_mw\|_{L^{p,\lm}}
=
\|f_mw\|_{\cM^{p_0}_p}
\le
\|w\1_{mQ_m}\|_{\cM^{q_0}_q}
\|f_m\|_{\cM^{n/\al}_r}.
\end{equation}
Since, 
\[
\|f_m\|_{\cM^{n/\al}_r}
\le
\|f_m\|_{L^{n/\al}}
\approx
(\log m)^{\al/n},
\]
the inequalities 
\eqref{3.2}--\eqref{3.4} 
yield the contradiction 
$(\log m)^{1-\al/n}\le C$. 
Thus, the statement \text{(c)} holds. 

\subsection{Proof of Theorem \ref{thm1.5} (II)}\label{ssec3.2}
Assume the statements \text{(b)} and \text{(c)}. 
To prove sufficiency we may assume that 
the function $f$ is non-negative and bounded and 
that $f$ has compact support.  
Fix $Q_0\in\cQ$. 
We shall evaluate the quantity 
\[
\lt(\frac1{|Q_0|^{\lm/n}}\int_{Q_0}[(I_{\al}f)w]^q\,dx\rt)^{1/q}.
\]
By Lemma \ref{lem3.1} 
we can select a $1/2$ sparse family 
$\cS\subset\cD(Q_0)$ 
such that, for all $x\in Q_0$, 
\[
I_{\al}f(x)
\le C
\lt(
I_{\al}^{\cS}f(x)
+
C_{\8}
\rt),
\]
where 
\[
I_{\al}^{\cS}f(x)
=
\sum_{Q\in\cS}
\1_{Q}(x)|Q|^{\al/n}\fint_{3Q}f\,dy
\]
and 
\[
C_{\8}
=
\sum_{k=0}^{\8}
|\kp^kQ_0|^{\al/n}\fint_{\kp^kQ_0}f\,dx.
\]
It follows from \eqref{1.9} that 
\begin{align*}
C_{\8}\lt(\frac1{|Q_0|^{\lm/n}}\int_{Q_0}w^q\rt)^{1/q}
&\le
C_{\8}\|w\1_{Q_0}\|_{L^{q,\lm}}
\\ &\le 
\sum_{k=0}^{\8}
2^{-k}
|\kp^kQ_0|^{\al/n-1}
\|w\1_{\kp^kQ_0}\|_{L^{q,\lm}}
\int_{\kp^kQ_0}f\,dx,
\end{align*}
by H\"{o}lder's inequality \eqref{1.4}, 
that
\[
\int_{\kp^kQ_0}f\,dx
\le C
\|w^{-1}\1_{\kp^kQ_0}\|_{H^{p',\lm}}
\|fw\1_{\kp^kQ_0}\|_{L^{p,\lm}},
\]
by the use of Theorem \ref{thm1.4} (I),
and that
\[
C_{\8}\lt(\frac1{|Q_0|^{\lm/n}}\int_{Q_0}w^q\rt)^{1/q}
\le C C_2
\|fw\|_{L^{p,\lm}}
\sum_{k=0}^{\8}2^{-k}
=C C_2
\|fw\|_{L^{p,\lm}}.
\]

Let $u=w^q$. We wish to estimate 
$\|I_{\al}^{\cS}f\|_{L^q(u)}$
by way of a duality argument. 
To this end, we take a function $g$, 
which is non-negative, supported in $Q_0$ and 
satisfies $\|g\|_{L^{q'}(u)}=1$, 
and evaluate the quantity
\[
\text{(i)}
=
\sum_{Q\in\cS}
|Q|^{\al/n}
\fint_{3Q}f\,dy
\int_{Q}g\,du.
\]

By the statement \text{(b)}, 
\begin{align*}
|Q|^{\al/n}
\lt(\frac{w^{-1}(Q)}{|Q|}\rt)
\cdot
|Q|^{1/q_0}
\lt(\frac{u(Q)}{|Q|}\rt)^{1/q}
&\le
\|(M_{\al}[w^{-1}\1_{Q}])w\|_{L^{q,\lm}}\\
&\le C_2
\|\1_{Q}\|_{L^{p,\lm}}\\
&=C_2
|Q|^{1/p_0}.
\end{align*}
Since $1/q_0=1/p_0-\al/n$, 
\[
\lt(\frac{u(Q)}{|Q|}\rt)
\lt(\frac{w(Q)}{|Q|}\rt)^q
\le C_2^q,
\]
which means that 
$u$ belongs to $A_{q+1}$ with the estimate
$[u]_{A_{q+1}}\le C_2^q$. 
This and Lemma \ref{lem2.2} give us that 
\begin{equation}\label{3.5}
u(Q)\le C C_2^qu(E_{Q})
\text{ for all }Q\in\cS.
\end{equation}

It follows from \eqref{3.5} that 
\begin{align*}
\text{(i)}
&=
\sum_{Q\in\cS}
|Q|^{\al/n}
\fint_{3Q}f\,dy
\frac{\int_{Q}g\,du}{u(Q)}
u(Q)
\\ &\le C C_2^q
\sum_{Q\in\cS}
|Q|^{\al/n}
\fint_{3Q}f\,dy
\frac{\int_{Q}g\,du}{u(Q)}
u(E_{Q})^{1/q+1/q'}
\\ &\le C C_2^q
\lt\{
\sum_{Q\in\cS}
\lt(
|Q|^{\al/n}
\fint_{3Q}f\,dy
\rt)^qu(E_{Q})
\rt\}^{1/q}
\\ &\quad\times
\lt\{
\sum_{Q\in\cS}
\lt(\frac{\int_{Q}g\,du}{u(Q)}\rt)^{q'}
u(E_{Q})
\rt\}^{1/q'}
\\ &\le C C_2^q q
\lt(\int_{Q_0}[(M_{\al}f)w]^q\,dx\rt)^{1/q},
\end{align*}
where in the last inequality 
we have used the $L^{q'}(u)$ boundedness of 
the dyadic weighted Hardy-Littlewood maximal operator 
with respect to $u$ with the norm less than or equal to $q$;
see (\ref{5.2}).

This and the statement \text{(b)} yield 
\[
\lt(\frac1{|Q_0|^{\lm/n}}\int_{Q_0}[(I_{\al}^{\cS}f)w]^q\,dx\rt)^{1/q}
\le C C^{q+1}
\|fw\|_{L^{p,\lm}},
\]
which completes the proof.

\section{The power weight cases and some equivalences}\label{sec4}
In this section 
we investigate the case of the power weight cases and 
introduce some equivalence conditions for our theorems for the purpose.

We first give the certain range of the power for which the 
boundedness of $M_{\al}$ on power weighted Morrey spaces as follows.

\begin{prop}\label{pr-power}
Suppose that the 
parameters satisfy the same conditions as in Theorem \ref{thm1.4} 
and let 
$w_{\rho}(x)=|x|^{\rho}$ with $\rho>-n$.  
Then the following are equivalent. 
\begin{itemize}
\item[\text{(a)}] 
There exists a constant $C_1>0$ 
such that 
\[
\|(M_{\al}f)w_{\rho}\|_{L^{q,\lm}}
\le C_1
\|fw_{\rho}\|_{L^{p,\lm}}
\]
holds for every function $f$ with 
$fw_{\rho}\in L^{p,\lm}(\R^n)$ 
\item[\text{(b)}]
There exist a constant $C_2>0$ such that 
\[
\sup_{Q\in\cQ}
|Q|^{\al/n-1}
\|w_{\rho}\1_{Q}\|_{L^{q,\lm}}
\|w_{\rho}^{-1}\1_{Q}\|_{H^{p',\lm}}
\le C_2.
\]
\item[\text{(c)}]
The parameter $\rho$ satisfies
\[
-n+\lm\le q\rho,\quad
p\rho<n(p-1)+\lm.
\]
\end{itemize}
\end{prop}

Since the proof of Proposition \ref{pr-power} 
is almost the same as the one of 
\cite[Proposition 4.2]{Ta}, we omit the proof. 
Meanwhile, as we observed in \cite{NS15}, the condition \eqref{1.9} with the power weight 
$w=w_{\rho}$ is equivalent to 
$q\rho>-n+\lm$. 
Hence, we obtain the power weight result for $I_{\al}$ as follows. 

\begin{prop}
Suppose the parameters satisfy the same conditions as in Theorem \ref{thm1.5} 
and let 
$w_{\rho}(x)=|x|^{\rho}$ with $\rho>-n$.  
Then the following are equivalent. 
\begin{itemize}
\item[\text{(a)}] 
There exists a constant $C_1>0$ 
such that 
\[
\|(I_{\al}f)w_{\rho}\|_{L^{q,\lm}}
\le C_1
\|fw_{\rho}\|_{L^{p,\lm}}
\]
holds for every function $f$ with 
$fw_{\rho}\in L^{p,\lm}(\R^n)$ 
\item[\text{(b)}]
There exist constants $C_2>0$ such that 
\[
|Q|^{\al/n-1}
\|w_{\rho}\1_{Q}\|_{L^{q,\lm}}
\|w_{\rho}^{-1}\1_{Q}\|_{H^{p',\lm}}
\le C_2
\text{ and }
2\|w_{\rho}\1_{Q}\|_{L^{q,\lm}}
\le
\|w_{\rho}\1_{\kp Q}\|_{L^{q,\lm}}
\]
hold for some $\kp>1$ and all $Q\in\cQ$.
\item[\text{(c)}]
The parameter $\rho$ satisfies
\[
-n+\lm<q\rho,\quad
p\rho<n(p-1)+\lm.
\]
\end{itemize}
\end{prop} 

Finally, we note one observation. 
As we mentioned in Theorem \ref{thm1.4}, 
the weight problem for the maximal operator $M$
is still open.  
One finds
that the problem is difficult
since
it is difficult to calculate the quantities 
\[
\|w\1_{Q}\|_{L^{q,\lm}},
\quad
\|w^{-1}\1_{Q}\|_{H^{p',\lm}}
\]
appearing in \eqref{1.6}.  
Indeed, in the Lebesgue setting 
$p_0=p$ and $q_0=q$, 
it is easy to calculate these quantities.
Thus,
it is important to calculate these quantities
when $p_0 \ne p$ and $q_0 \ne q$. 
At least, we have the explicit formula for the quantity 
$
\|w\1_{Q}\|_{L^{q,\lm}}
$ as follows. 

\begin{prop}\label{prop4.1}
Let 
$0\le\al<n$, 
$1<p<p_0<\8$, 
$1<q<q_0<\8$ 
and $w$ be a weight. 
Suppose that 
\[
\frac1{q_0}=\frac1{p_0}-\frac{\al}n
\text{ and }
\frac{q}{q_0}=\frac{p}{p_0}.
\]
Set 
$\lm/n=1-p/p_0=1-q/q_0$. 
If we assume Theorem \ref{thm1.4} \text{(b)}, 
then, for all $Q\in\cQ$, 
\[
\|w\1_{Q}\|_{L^{q,\lm}}
\approx
|Q|^{1/q_0}
\lt(\fint_{Q}w^q\,dx\rt)^{1/q}.
\]
That is, then the Morrey norm 
is attained on the full cube $Q$.
\end{prop}

\begin{proof}
The relation 
\[
\|w\1_{Q}\|_{L^{q,\lm}}
\ge
|Q|^{1/q_0}
\lt(\fint_{Q}w^q\,dx\rt)^{1/q}
\]
follows automatically. 
We shall prove the converse. 

Let $u=w^q$, $\bt=nq/q_0$ 
and 
$\gm=|Q|^{\bt/n}\fint_{Q}u\,dx$.
Consider 
\[
\Om
=
\{x\in Q:\,
M_{\bt}[u\1_{Q}](x)
>
2\cdot3^n\gm
\}.
\]
Since, for all $x\in Q$, 
\[
M_{\bt}[u\1_{Q}](x)
\le
|Q|^{\bt/n}
M[u\1_{Q}](x),
\]
we have that 
\[
\Om
\subset
\lt\{x\in Q:\,
M[u\1_{Q}](x)
>
\frac{2\cdot3^n\gm}{|Q|^{\bt/n}}
\rt\}.
\]
The weak-$(1,1)$ boundedness of $M$ gives us that 
\[
|\Om|
\le 3^n
\frac{|Q|^{\bt/n}}{2\cdot3^n\gm}
\int_{Q}u
=
\frac12|Q|.
\]
Hence, if we let 
$E=Q\setminus\Om$, we have 
$|E|\ge|Q|/2$. 
It follows from H\"{o}lder's inequality \eqref{1.5} that 
\[
\frac12|Q|
\le|E|\le C
\|w\1_{E}\|_{L^{p,\lm}}
\|w^{-1}\1_{E}\|_{H^{p',\lm}}.
\]
Because we always have 
\[
\|w^{-1}\1_{E}\|_{H^{p',\lm}}
\le
\|w^{-1}\1_{Q}\|_{H^{p',\lm}}
\le C_2
\lt(|Q|^{\al/n-1}\|w\1_{Q}\|_{L^{q,\lm}}\rt)^{-1},
\]
we have that 
\[
\|w\1_{Q}\|_{L^{q,\lm}}
\le C|Q|^{-\al/n}
\|w\1_{E}\|_{L^{p,\lm}}.
\]
By the definition of the Morrey norm, 
using 
$1/p_0-1/q_0=\al/n>0$ 
and $q>p$, we see that 
\[
\|w\1_{E}\|_{L^{p,\lm}}
\le
|Q|^{1/p_0-1/q_0}
\|w\1_{E}\|_{L^{q,\lm}}.
\]
Thus, noticing 
$1/p_0-1/q_0-\al/n=0$ 
and the fact that 
\[
\|w\1_{E}\|_{L^{q,\lm}}
\le
(2\cdot3^n\gm)^{1/q},
\]
we conclude that 
\[
\|w\1_{Q}\|_{L^{q,\lm}}
\le C
|Q|^{1/q_0}
\lt(\fint_{Q}u\,dx\rt)^{1/q},
\]
which proves the proposition. 
\end{proof}

\section{Appendix--Universal estimates}

Let $\mu$ be a Radon measure on ${\mathbb R}^n$.
An example we envisage here is the weighted measure
$\mu=w\,dx$.
We consider the following dyadic weighted Hardy-Littlewood maximal operator:
\[
M_{{\rm dyadic},w}f(x)=
\sup_{Q}\frac{\chi_Q(x)}{w(Q)}\int |f|w\,dy,
\]
where $Q$ moves over all dyadic cubes in $\R^n$.
Using a covering lemma,
we can prove
\[
w\{x \in \R^n\,:\,M_{\rm dyadic}f(x)>\lm\}
\le 
\frac{1}{\lm}\|f\chi_{\{x \in \R^n\,:\,M_{\rm dyadic}f(x)>\lm\}}\|_{L^1(w)},
\]
which yields
\begin{equation}\label{5.2}
\|M_{{\rm dyadic},w}f\|_{L^p(w)}
\le p'
\|f\|_{L^p(w)}.
\end{equation}
We consider the following weighted dyadic Morrey norm:
\[
\|f\|_{L^{p,\lm}_{\rm dyadic}(w)}
=
\sup_{Q}
\lt(\frac1{w(Q)^{\lm/n}}\int_{Q}|f|^pw\,dx\rt)^{1/p},
\]
where $Q$ moves over all dyadic cubes in $\R^n$.

\begin{lem}[Universal estimate for Morrey spaces]\label{lem:5.2}
Let $1<p<\8$ and $0<\lm<n$.
Then
\[
\|M_{{\rm dyadic},w}f\|_{L^{p,\lm}_{\rm dyadic}(w)}
\le (p'+1)
\|f\|_{L^{p,\lm}_{\rm dyadic}(w)}.
\]
\end{lem}

\begin{proof}
Fix a dyadic cube $Q$.
Then it suffices to show
\begin{equation}\label{5.3}
\lt(\frac1{|Q|^{\lm/n}}\int_{Q}(M_{{\rm dyadic},w}f)^p\,dx\rt)^{1/p}
\le (p'+1)
\|f\|_{L^{p,\lm}_{\rm dyadic}(w)}.
\end{equation}
To this end, we decompose $f$
according to $Q$.
Let $f_1=f\chi_Q$ and $f_2=f-f_1$.
Then 
from (\ref{5.2}) we have
\begin{equation}\label{5.4}
\lt(\frac1{|Q|^{\lm/n}}\int_{Q}(M_{{\rm dyadic},w}f_1)^p\,dx\rt)^{1/p}
\le p'
\|f\|_{L^{p,\lm}_{\rm dyadic}(w)}
\end{equation}
and
from the pointwise equality
\begin{equation}\label{5.5}
M_{{\rm dyadic},w}f_2(x)
=
\inf_{y \in Q}
M_{{\rm dyadic},w}(f\chi_{{\mathbb R}^n \setminus Q})(y)
\end{equation}
we have
\begin{equation}\label{5.6}
\lt(\frac1{|Q|^{\lm/n}}\int_{Q}(M_{{\rm dyadic},w}f_2)^p\,dx\rt)^{1/p}
\le 
\|f\|_{L^{p,\lm}_{\rm dyadic}(w)}.
\end{equation}
Combining (\ref{5.4}) and (\ref{5.6}),
we conclude (\ref{5.3}).
\end{proof}

\end{document}